\documentclass[reqno,a4paper,12pt]{amsart} 
\usepackage{amsmath,amscd,amsfonts,amssymb}
\usepackage{mathrsfs,dsfont}
\usepackage{hyperref}

\usepackage{tikz,graphics}
\usetikzlibrary{arrows.meta, 
	bending,     
	patterns     
}

\numberwithin{equation}{section}
\numberwithin{figure}{section}

\newcommand\R{\mathbb{R}}

\newcommand\Z{\mathbb{Z}}

\newcommand\T{\mathbb{T}}

\newcommand\lam{\lambda}
\newcommand\Lam{\Lambda}
\newcommand\Sig{\Sigma}

\newcommand\Om{\Omega}

\newcommand\1{\mathds{1}}
\newcommand\eps{\varepsilon}

\renewcommand\le{\leqslant}
\renewcommand\ge{\geqslant}
\renewcommand\leq{\leqslant}
\renewcommand\geq{\geqslant}
\newcommand\sbt{\subset}

\newcommand{\ft}[1]{\widehat #1}
\newcommand{\dotprod}[2]{\langle #1 , #2 \rangle}
\newcommand{\mes}{\operatorname{mes}}

\newcommand{\supp}{\operatorname{supp}}

\newcommand{\zoft}[1]{{Z(\ft{\1}_{#1})}}
\newcommand{\zfft}{{Z(\ft{f}\,)}}
\newcommand{\cm}{\complement}

\newcommand{\interior}{\operatorname{int}}

\theoremstyle{plain}
\newtheorem{thm}{Theorem}[section]
\newtheorem{lem}[thm]{Lemma}

\newtheorem{prop}[thm]{Proposition}
\newtheorem{conj}[thm]{Conjecture}

\newtheorem*{claim*}{Claim}

\newcommand{\thmref}[1]{Theorem~\ref{#1}}
\newcommand{\secref}[1]{Section~\ref{#1}}

\newcommand{\lemref}[1]{Lemma~\ref{#1}}

\newcommand{\propref}[1]{Proposition~\ref{#1}}

\newcommand{\conjref}[1]{Conjecture~\ref{#1}}
\newcommand{\exampref}[1]{Example~\ref{#1}}

\theoremstyle{definition}
\newtheorem{definition}[thm]{Definition}
\newtheorem{example}[thm]{Example}
\newtheorem*{definition*}{Definition}
\newtheorem*{remarks*}{Remarks}
\newtheorem*{remark*}{Remark}
\newtheorem{remark}[thm]{Remark}

\newenvironment{enumerate-roman}
{\begin{enumerate}
\addtolength{\itemsep}{5pt}
}
{\end{enumerate}}

\newenvironment{enumerate-alph}
{\begin{enumerate}
\addtolength{\itemsep}{5pt}
}
{\end{enumerate}}

\newenvironment{enumerate-num}
{\begin{enumerate}
\addtolength{\itemsep}{5pt}
}
{\end{enumerate}}

\newenvironment{enumerate-text}
{\begin{enumerate}
\addtolength{\itemsep}{5pt}
}
{\end{enumerate}}

\newcommand{\beql}[1]{\begin{equation}\label{#1}}
\newcommand{\eeq}{\end{equation}}
\newcommand{\comment}[1]{}

\newcommand{\RR}{{\mathbb R}}

\newcommand{\ZZ}{{\mathbb Z}}

\newcommand{\one}{{\mathds{1}}}

\usepackage{txfonts,pxfonts,tikz} 
\usepackage{tgschola}
\usepackage[T1]{fontenc}

\makeatletter
\@namedef{subjclassname@2020}{\textup{2020} Mathematics Subject Classification}
\makeatother

\begin{document}

\title{Spectral sets and weak tiling}

\author[M. N. Kolountzakis]{{Mihail N. Kolountzakis}}
\address{M.K.: Department of Mathematics and Applied Mathematics, University of Crete, Voutes Campus, GR-700 13, Heraklion, Crete, Greece}
\email{kolount@uoc.gr}

\author[N. Lev]{Nir Lev}
\address{N.L.: Department of Mathematics, Bar-Ilan University, Ramat-Gan 5290002, Israel}
\email{levnir@math.biu.ac.il}

\author[M. Matolcsi]{M\'at\'e Matolcsi}
\address{M.M.: Department of Analysis, Institute of Mathematics, Budapest University of Technology and Economics (BME), M\"{u}egyetem rkp. 3., H-1111 Budapest, Hungary (also at Alfred Renyi Institute of Mathematics, ELKH,  H-1053, Realtanoda u 13-15, Budapest, Hungary)}
\email{matomate@renyi.hu}

\date{August 30, 2023}
\thanks{M.K.\ was supported by the Hellenic Foundation for Research and Innovation, Project HFRI-FM17-1733 and by University of Crete Grant 4725.}
\thanks{N.L.\ was supported by ISF Grant No.\ 1044/21 and 
ERC Starting Grant No.\ 713927.}
\thanks{M.M.\ was supported by NKFIH grants K129335 and K132097.}

\begin{abstract}
A set $\Omega \subset \mathbb{R}^d$ is said to be spectral if the space $L^2(\Omega)$ admits an orthogonal basis of exponential functions. Fuglede (1974) conjectured that $\Omega$ is spectral if and only if it can tile the space by translations. While this conjecture was disproved for general sets, it was recently proved that the Fuglede conjecture does hold for the class of convex bodies in $\mathbb{R}^d$. The proof was based on a new geometric necessary condition for spectrality, called ``weak tiling''. In this paper we study further properties of the weak tiling notion, and present applications to convex bodies, non-convex polytopes, product domains and Cantor sets of positive measure.
\end{abstract}

\maketitle


\section{Introduction}

\subsection{Spectral sets}
Let $\Om \subset \R^d$ be a bounded, measurable set of positive
measure. We say that $\Om$ is \emph{spectral} if there exists a countable set 
$\Lambda\subset \R^d$  such that the system of exponential functions
$\{\exp 2\pi i\dotprod{\lambda}{x}\}$, $\lambda\in \Lambda$,
forms an orthogonal basis in  $L^2(\Om)$,
  that is, the system is orthogonal and complete in the space.
In this case the set  $\Lambda$ is called a \emph{spectrum} for $\Om$.

The study of spectral sets has a long history, which goes back to
Fuglede \cite{Fug74} who conjectured that the spectral sets
could be characterized geometrically as the sets which 
can tile the space by translations. 
We say that the set $\Om$ \emph{tiles the space by translations} 
if there exists a countable set $\Lambda\subset \R^d$ such that the 
translated copies  $\{\Om + \lam\}$, $\lam \in \Lam$,
constitute a partition of $\R^d$ up to measure zero. 

Fuglede's conjecture inspired extensive 
research over the years, and a number of interesting results establishing
connections between spectrality and tiling
had since been obtained. On the other hand,
also counterexamples were
found to both directions of the conjecture 
in dimensions $d \geq 3$, see \cite{Tao04},
 \cite[Section 4]{KM10}.

It was recently proved in \cite{LM22} that the Fuglede conjecture
holds for the class of \emph{convex bodies} in $\R^d$ (compact, 
convex sets with nonempty interior). That is, a convex body
$\Om \subset \R^d$ is spectral if and only if $\Om$ can
tile the space by translations.

\subsection{Weak tiling}
To prove  the Fuglede conjecture for convex domains,
 the authors in \cite{LM22} established a  link between
 the analytic notion of 
spectrality and a geometric notion which was
termed ``weak tiling''. It is defined as follows:

\begin{definition}
We say that the set $\Om$ can \emph{weakly tile}
another measurable (possibly unbounded) 
set $\Sig \subset \R^d$ by translations,
if there exists a positive, locally finite Borel measure $\nu$
on $\R^d$ such that $\1_{\Om} \ast \nu = \1_{\Sig}$ a.e.
\end{definition}

Note that in the special case where the
measure $\nu$ is the sum of (finitely or countably many)
unit masses,  the 
weak tiling becomes a \emph{proper tiling}
of $\Sigma$ by translates of $\Om$.

The following result was proved in \cite{LM22}.

\begin{thm}[{see \cite[Theorem 1.5]{LM22}}]
\label{thmLMSWT}
Let $\Om$ be a bounded, measurable set  in $\R^d$. If $\Om$ is spectral,
then it can weakly tile its complement $\Om^\cm = \R^d \setminus \Om$
  by translations.
\end{thm}

This result thus gives a geometric condition
necessary for spectrality. It establishes a weak form of the
``spectral implies tiling'' part of  Fuglede's conjecture.
We observe that the weak tiling conclusion cannot in
general be strengthened to proper tiling 
since there exist examples of spectral sets which cannot tile
by translations.

Several applications of \thmref{thmLMSWT} were given
in \cite{LM22}. The main one is the proof that the Fuglede conjecture
holds for convex bodies in $\R^d$. As another  application,
 it was proved in \cite[Theorem 3.6]{LM22}  that
 if a bounded, open  set  $\Om \sbt \R^{d}$ 
 is spectral  then its boundary $\partial \Om$
  must be a set of Lebesgue measure zero.

In the present paper we study further properties and applications
of weak tiling.

\subsection{Equidecomposability}
Let $A,B$ be two (not necessarily convex)
polytopes in $\R^d$. We say that $A$ and $B$ are
 \emph{equidecomposable} if $A$  can be partitioned,
 up to measure zero, into a finite number of smaller polytopes which can be 
rearranged using rigid motions to form, again up to measure zero, a 
partition of $B$. If the pieces of the partition can be
rearranged  using translations only, then 
 $A$ and $B$ are said to be \emph{equidecomposable by translations}.

It has long been known that if a polytope $A \subset \R^d$ can 
(properly) tile the 
space by translations, then  $A$ must be equidecomposable by translations to a cube
of the same volume. This result was first established 
by {M\"urner} in \cite{Mur75}, and
was later rediscovered in \cite{LM95a}.
Recently, it was proved in \cite{LL21}
that the same conclusion holds also for the class of spectral
polytopes, that is, any spectral polytope $A \sbt \R^d$
must be equidecomposable by translations
to a cube of the same volume.

We will prove the following simultaneous
strengthening of the  latter two results:

\begin{thm}
\label{thmNCWT}
    Let $A$ be a (not necessarily convex)
    polytope in $\R^d$. Assume that $A$
    can weakly tile its complement by
    translations. Then $A$ is equidecomposable 
    by translations to a cube of the same volume.
\end{thm}

\subsection{Convex domains}
\label{secINTCVXDOM}

It is a classical result due to {M\"urner},
see \cite[Section 3.3]{Mur77},
that a \emph{convex} polytope $A \sbt \R^d$ is
equidecomposable by translations to a cube if and only if
$A$ is centrally symmetric and all the
facets of $A$ (that is, the $(d-1)$-dimensional faces of $A$)
are also centrally symmetric. Together with
 \thmref{thmNCWT} this implies that
a convex polytope which
    can weakly tile its complement  by
    translations must be centrally symmetric 
    and have centrally symmetric facets.

Actually, this conclusion can be significantly
strengthened as follows.
It was proved in \cite[Theorem 4.1]{LM22}
that if $A$ is a general convex body 
in $\R^d$ and if $A$ can weakly tile its complement   by 
translations, then $A$ must in fact be a convex polytope.
 It was also proved \cite[Theorem 6.1]{LM22} that
  if, in addition,
 $A$ is centrally symmetric and has centrally symmetric facets, 
then each belt of $A$  must have either $4$ or $6$ facets. 
So in the latter case, it follows by the 
Venkov-McMullen theorem
\cite{Ven54}, \cite{McM80} (see also
 \cite[Section 32.2]{Gru07})
that $A$ can tile its complement
 not only weakly, but even
 properly, by translations.
We thus arrive at the following result:

\begin{thm}
\label{thmCVXBWT}
    Let $A$ be a convex body in $\R^d$, and assume that
    $A$ can weakly tile its complement by
    translations. Then $A$ must be a convex
 polytope which can also tile the space properly by translations.
\end{thm}

We observe that for general sets (i.e.\ not assumed to be convex)
the result does not hold, 
since there exist examples of spectral sets which cannot tile
by translations.

 \subsection{Product domains}
Let $A \subset \R^n$ and  $B \subset \R^m$ be two
bounded, measurable sets. It is known that if
$A$ and $B$ are both spectral  sets in $\R^n$
and $\R^m$ respectively,
then their cartesian product 
 $\Omega = A \times B$  is spectral  in 
$\R^n \times \R^m$.
 Indeed, if $U \subset \R^n$  is a spectrum for $A$,  
and $V \subset \R^m$ is a spectrum for $B$,  then the
 product set $\Lam = U \times V$  serves as 
 a spectrum for $\Om$ (see e.g.\ \cite[Theorem 3]{JP99}).

In  \cite{Kol16} the question was posed as  to whether
the converse statement is also true. 

\begin{conj}
	\label{conjA1.1}
	Let $A \subset \R^n$ and $B \subset \R^m$   be two
	bounded, measurable sets. Then their product  $\Omega  = A \times B$
	is spectral if and only if $A$ and $B$ are both spectral sets.
\end{conj}

The ``only if'' part of this conjecture  is the non-trivial one. 
The difficulty lies in that we assume the product set $\Omega$ to be spectral, but
 we do not make any a priori assumption
 that the spectrum $\Lambda$ also has a  product structure,
so it is not obvious which sets $U$ and $V$ may serve  as spectra for the factors $A$ and $B$, respectively.
(One can show, see \cite[Lemma 2]{JP99}, that
if  $\Om$ happens to admit  a spectrum $\Lam$ with a product structure, $\Lam = U \times V$, then $U$ is a spectrum for $A$ and $V$ is a spectrum for $B$.)

It was proved  in \cite{GL16}
that \conjref{conjA1.1}
holds in the case   where one of the factors, say $A$,
 is an interval in $\R$. In \cite{Kol16} 
 it was established, using a different approach,
that the conjecture is true also if the set $A$ is the union
 of two intervals in $\R$. 
In \cite{GL20} the conjecture was proved in the case where
the factor $A$ is a convex polygon in $\R^2$.

As an application of the weak tiling method, we will prove
the following result:

\begin{thm}
	\label{thmPRDCVX}
	Let $\Omega=A\times B$ where $A$ is a convex body in $\R^n$, 
	while $B$ is any bounded, measurable set in $\R^m$.
	If $\Omega$ is a spectral set then $A$ must be spectral,
    or equivalently, $A$ must be a convex
 polytope which can (properly) tile the space by translations.
\end{thm}

This significantly strengthens \cite[Theorems 2.1 and 2.2]{GL20}
where it was proved that $A$ cannot have a smooth boundary,
and that if $A$ is assumed a priori to be a convex polytope,
then $A$ must be centrally symmetric and have centrally symmetric facets.

It is known, 
see \cite[Section~1.2]{Kol16}, that 
 the product set  $\Omega  = A \times B$ can tile the 
 space $\R^n \times \R^m$ by translations if and 
 only if both $A$ tiles $\R^n$ and $B$ tiles $\R^m$.
In order to prove  \thmref{thmPRDCVX} we shall use
 an analogous result for weak tiling:

\begin{thm}
	\label{thmPRDWKT}
	Let $A \sbt \R^n$ and $B \sbt \R^m$ be two bounded,
 measurable sets. Then the product set
 $\Omega=A\times B$ can weakly tile
 its complement in $\R^n \times \R^m$ by translations
 if and only if 
  both $A$ and $B$ weakly tile
 their complements in $\R^n$ and $\R^m$
 respectively.
 \end{thm}

\thmref{thmPRDCVX} is thus obtained by a combination of
Theorems \ref{thmLMSWT}, \ref{thmPRDWKT} and
\ref{thmCVXBWT}.

Similarly, by a combination of 
Theorems \ref{thmLMSWT}, \ref{thmPRDWKT} and
\ref{thmNCWT} we obtain that
if $A$ is a (not necessarily convex)
    polytope in $\R^n$, and if $B$ is any
     bounded,
 measurable set in $\R^m$, then the spectrality
of the product $\Omega=A\times B$ 
implies that $A$ is equidecomposable 
    by translations to a cube.
    Alternatively, it is possible 
     to obtain this result
    by a combination of 
    \cite[Lemma 5.1]{GL20} and 
    \cite[Theorem 7.1]{LL21}.

\subsection{Nowhere dense sets}
\label{subsecNW}
Assume now that $\Om \sbt \R^d$ is
a \emph{bounded, nowhere dense set} 
of positive measure. It is not hard to show that
such a set $\Om$ cannot tile the space by translations.
To see this, suppose to the contrary that $\Omega+\Lambda$ is a tiling, 
then $\Lambda$ must be a locally finite set. Since $\Om$ is bounded,
 any ball  can thus intersect only finitely many translated copies $\Omega+\lambda$, $\lambda \in \Lambda$. Since $\Omega$ is a nowhere dense set, then also any finite union of translates of $\Omega$ is a nowhere dense set. Hence the union of all the translated copies $\Omega+\lambda$ intersecting a given ball  is a nowhere dense set, so this union does not cover a set of full measure in the ball, a contradiction.

In \cite{Mat05}, the following question was 
posed: can a bounded, nowhere dense set
 be spectral?
The answer is expected to be negative, but so far this has been proved only in dimension one. 
In fact, it was proved in \cite[Corollary 1.6]{IK13}
that if a bounded, measurable set $\Omega \sbt \R$ is 
spectral, then it
can \emph{multi-tile} the real line by translations.
This  excludes the possibility that $\Om$ is a
nowhere dense set, by the same argument as above.

In this paper we use the weak tiling method
as a different approach to the spectrality problem
for bounded, nowhere dense sets. We will prove the following result:

\begin{thm}
\label{thmCSNWT}
Let $E \sbt \R$ be a symmetric Cantor set
of positive measure.  Then $E$  
cannot weakly tile its   complement by translations.
As a consequence, $E$ is not spectral.
\end{thm}

The definition
of a symmetric Cantor set
will be given in
\secref{secCANTOR}.

If we combine \thmref{thmCSNWT} with \thmref{thmPRDWKT}
then we obtain the following conclusion for multi-dimensional
nowhere dense sets with a cartesian product structure:

\begin{thm}
	\label{thmPRDCSPM}
Let $\Omega=E\times B$ be the product
of a symmetric Cantor set $E \sbt \R$ 
of positive measure, and an arbitrary
bounded, measurable set $B \sbt \R^m$.
Then  $\Omega$ is not a spectral set
in $\R \times \R^m$.
\end{thm}

We give two proofs of \thmref{thmCSNWT}.
In \secref{secCANTOR} we prove it by closely examining the \textit{essential difference set} of a symmetric Cantor set in $\R$ and proving that the Cantor set admits a \textit{packing region} which is strictly larger than the set itself (see \secref{secPACKREG}). In \secref{secOVERLAPS} we prove directly that the  Cantor set cannot weakly tile its complement by translations, by establishing in a quantitative way the same property for the $n$-th generation set that approximates the Cantor set.


\section{Weak tiling and equidecomposability}\label{secEQUID}

In this section we prove \thmref{thmNCWT},
that is, we show that if a (not necessarily convex)
    polytope $A \sbt  \R^d$
    can weakly tile its complement by
    translations, then $A$ is equidecomposable 
    by translations to a cube.
As a consequence, any \emph{convex} polytope
    that can weakly tile its complement  by
    translations  must be centrally symmetric 
    and have centrally symmetric facets.
    In turn, this implies \thmref{thmCVXBWT}
    (see \secref{secINTCVXDOM}).

\subsection{}
The Fourier transform of a function $f \in L^1(\R^d)$ is defined by
\[
\ft f (\xi)=\int_{\R^d} f (x) \, e^{-2\pi i\langle \xi,x\rangle} dx,
\quad \xi \in \R^d.
\]
We will need a result from \cite{LL21}
concerning the zeros of the Fourier transform
$\ft{\1}_A$ of the indicator function $\1_A$ of a 
(not necessarily convex) polytope $A \subset \R^d$.

Given  $\delta>0$ and
real numbers $\tau_1, \dots, \tau_k$
we consider the set
\begin{equation}
\label{eqTdef}
    T = T(\delta; \tau_1, \dots, \tau_k) =
    \{ n \in \Z : |e^{2 \pi i n \tau_j} - 1| < \delta, 
    \; 1 \leq j \leq k \}.
\end{equation}
If we are also given $R>0$ then we let
\begin{equation}
\label{eqJdef}
    J = J(R) = 
    \{ n \in \Z : 0 < |n| < R\}.
\end{equation}
Finally, given also
$\eps > 0$ and a vector $v \in \R^d$
we define
\begin{equation}
\label{eqSdef}
    S = S(T, J, v, \eps) = 
    \{ nv + w : n \in T \setminus J, \; w \in \R^d, \; |w|< \eps\}.
\end{equation}

The following result was proved in \cite{LL21}
although it was not explicitly stated there in this form.

\begin{thm}[{\cite{LL21}}]
\label{thmFA}
    Let $A$ be a polytope in $\R^d$. If
     $A$ is not equidecomposable by translations
    to a cube, then there exist $\delta>0$,
    real numbers $\tau_1, \dots, \tau_k$,
    a nonzero vector $v \in \R^d$,  $\eps>0$
    and $R>0$
    such that the Fourier transform $\ft{\1}_A$
    has no zeros in the set $S$, where $T$, $J$
    and $S$ are the three sets defined by \eqref{eqTdef},
     \eqref{eqJdef} and \eqref{eqSdef}.
\end{thm}

We briefly indicate how this 
can be inferred from \cite{LL21}.
If $A$ is not equidecomposable by translations
to any cube, then there exists a Hadwiger
functional $H_\Phi$ such that $H_\Phi(A) \neq 0$.
Let $p(u)$ be the trigonometric polynomial
constructed in \cite[Section 6.1]{LL21}, and
expand it  in the form $p(u) = \sum_{j=1}^{k} c_j
e^{2 \pi i u \tau_j}$. Let
$\delta = \delta(p, \eta) > 0$
be small enough such that
$|p(n) - p(0)| < \eta$ for every $n$
in the set \eqref{eqTdef}.
We then continue the proof as in 
\cite[Section 6]{LL21} to conclude
that there exist a nonzero vector
$v \in \R^d$,  $\eps>0$
    and $R>0$ such that $\ft{\1}_A$
has no zeros in the set \eqref{eqSdef}.

The results in \cite{LL21} are based on an
intricate analysis of
 the asymptotic behavior of the 
Fourier transform $\ft{\1}_A$,
see \cite[Theorem 4.1]{LL21}.
The analysis becomes considerably simpler
if the polytope $A$ is assumed to 
be convex, see \cite[Sections 3 and 4]{GL17}.

\subsection{}
A measure $\mu$ on $\R^d$  is said to be \emph{translation-bounded}
 if for every (or equivalently, for some) open ball $B$ we have 
\[
\sup_{x \in \R^d} |\mu|(B+x) < +\infty.
\]
If  $\mu$ is a translation-bounded measure on $\R^d$,
then $\mu$ is a tempered distribution. 

If $f$ is a function in $L^1(\R^d)$ and 
$\mu$ is a translation-bounded measure on $\R^d$,
then the convolution $f \ast \mu$ is 
a locally integrable function on $\R^d$
defined uniquely (up to equality a.e.)
by the condition that
$(f \ast \mu) \ast \varphi = 
f \ast (\mu \ast \varphi)$
for every continuous, compactly supported
function $\varphi$ on $\R^d$.

\begin{thm}
\label{thmC1}
Let $f\in L^1(\R^d)$, $\int f \neq 0$,
and let $\mu$ be
a translation-bounded measure on $\R^d$.
If we have $f \ast \mu = 1$ a.e., then
$\ft{\mu} = (\int f)^{-1} \cdot \delta_0$
in the open set
$\zfft^\complement$.
\end{thm}

Here we let $\zfft := \{ \xi \in \R^d : \ft{f}(\xi)=0\}$
denote the (closed) set of zeros of $\ft{f}$.

\thmref{thmC1} can be proved in a similar way to
\cite[Theorem 4.1]{KL16}.

\subsection{}
We now turn to the proof of the main result of this section.

\begin{proof}[Proof of \thmref{thmNCWT}]
Let $A$ be a polytope in $\R^d$, and assume that
$A$ can tile its complement weakly by
translations, that is, there exists a positive, locally
finite measure $\nu$ on $\R^d$ such that
$\1_A \ast \nu = \1_{A^\complement}$ a.e.
By \cite[Lemma 2.4]{LM22} the measure $\nu$
is not only locally finite, but in fact $\nu$
must be translation-bounded. It follows that the
measure $\mu := \delta_0 + \nu$ is translation-bounded
and satisfies $\1_A \ast \mu = 1$ a.e.
In turn, \thmref{thmC1}  implies that
$\ft{\mu} = m(A)^{-1} \cdot \delta_0$
in the open set ${\zoft{A}}^\complement$.

    We must prove that
    $A$ is equidecomposable by translations
    to a cube of the same volume.
    Suppose to the contrary that this is not the case.
    Then by \thmref{thmFA}
    there exist $\delta>0$,
    real numbers $\tau_1, \dots, \tau_k$,
    a nonzero vector $v \in \R^d$,  $\eps>0$
    and $R>0$ such that $\ft{\1}_A$
    has no zeros in the set $S$, where $T$, $J$
    and $S$ are the three sets defined by \eqref{eqTdef},
     \eqref{eqJdef} and \eqref{eqSdef}.
     It follows  that we have
    $\ft{\mu} = m(A)^{-1} \cdot \delta_0$
    in the open set $S$.

Now suppose that we are given a real-valued
Schwartz function
$g$ on $\R^d$ satisfying
\begin{equation}
\label{eqGC1}
    \supp(g) \sbt S,
    \quad
    \ft{g} \geq 0.
\end{equation}
Then we have
\begin{equation}
\label{eqGineq}
    \int_{\R^d} g(x) dx = \ft{g}(0) \leq 
    \ft{g}(0) + \int \ft{g}(\xi) d \nu(\xi) =  
     \int \ft{g}(\xi) d \mu(\xi),
\end{equation}
where the inequality in \eqref{eqGineq} is due to $\ft{g}$
being a nonnegative function and $\nu$
being a positive measure. On the  other hand,
we have
\begin{equation}
     \int \ft{g}(\xi) d \mu(\xi)= \mu(\ft{g}) =
     \ft{\mu}(g) = m(A)^{-1} g(0),
\end{equation}
where the last equality holds since we have
$\supp(g) \sbt S$ and
$\ft{\mu} = m(A)^{-1} \cdot \delta_0$
in the open set $S$. We conclude that
\begin{equation}
\label{eqGC2}
\int_{\R^d} g(x) dx \leq m(A)^{-1} g(0)
\end{equation}
for every real-valued Schwartz function $g$
satisfying \eqref{eqGC1}.
We will show that this leads to a contradiction,
by constructing an example of a
real-valued Schwartz function $g$ satisfying 
 \eqref{eqGC1}, but such that 
 \eqref{eqGC2} does not hold.
 
We choose a nonnegative Schwartz function $\varphi$
such that $\int \varphi =1$,
$\varphi$ is supported in
 the open ball of radius 
$\eps$ centered at the origin,
and $\ft{\varphi} \geq 0$. Next, let $\psi$
be a nonnegative, smooth function on the $k$-dimensional
torus $\T^k = (\R / \Z)^k$, such that
$\int \psi = 1$, $\psi$ is supported in the set
\begin{equation}
\label{eqPS1}
\{(t_1, \dots, t_k) : |e^{2 \pi i t_j} - 1| < \delta,
\; 1 \leq j \leq k\},
\end{equation}
  and  the Fourier coefficients
$\ft{\psi}(m)$, $m \in \Z^k$, are nonnegative.
We may assume that $R$ is an integer (by enlarging
it if necessary) and 
 define also the trigonometric polynomial
$p_N(t) := K_N(R \, t)$,
where $K_N$ is the classical Fej\'{e}r kernel,
\begin{equation}
\label{eqFej}
K_N(t) = \sum_{|n| < N} \Big(1 - \frac{|n|}{N} \Big)
e^{2 \pi i n t}, \quad t \in \T = \R / \Z.
\end{equation}
Then $p_N$ is nonnegative, $p_N(0) = N$, 
the Fourier coefficients $\ft{p}_N(n)$,
$n \in \Z$, are also nonnegative, 
$\ft{p}_N(0)=1$, and we
have  $\ft{p}_N(n)=0$ for every $n \in J$.

Finally, we define the function
\begin{equation}
\label{eqGN1}
g_N(x) := \sum_{n \in \Z}
\psi(n \tau) \,
\ft{p}_N(n) \, \varphi(x - n v),
\quad x \in \R^d,
\end{equation}
where we denote $\tau = (\tau_1, \dots, \tau_k)$.
Notice that there are only finitely many
nonzero terms in the sum \eqref{eqGN1}, and that the nonzero
terms correspond to integers $n$ belonging to 
the set $T \setminus J$.
Hence $g_N$ is a real-valued (in fact, nonnegative)
Schwartz function such that
$\supp(g_N)$ is contained in the set $S$. The Fourier transform
of $g$ is given by
\begin{equation}
\label{eqGN2}
\ft{g}_N(\xi) = \ft{\varphi}(\xi) \sum_{n \in \Z}
\psi(n \tau) \,
\ft{p}_N(n) \, e^{-2 \pi i n \dotprod{v}{\xi}}.
\end{equation}
Using the (absolutely convergent)
Fourier expansion of the function $\psi$, we obtain
\begin{equation}
\label{eqPS2}
\psi(n \tau) = \sum_{m \in \Z^k}
\ft{\psi}(m) e^{2 \pi i n \dotprod{m}{\tau}}.
\end{equation}
If we plug \eqref{eqPS2} into 
\eqref{eqGN2} and exchange the order of summation
(which is justified as $n$ goes through a finite
set of values) we obtain
\begin{equation}
\label{eqGN3}
\ft{g}_N(\xi) = \ft{\varphi}(\xi) 
\sum_{m \in \Z^k} \ft{\psi}(m) \,
p_N(\dotprod{m}{\tau} - \dotprod{v}{\xi}).
\end{equation}
Since all the terms on the right hand side
of \eqref{eqGN3} are nonnegative, we obtain
that $\ft{g}_N$ is a nonnegative function. We
conclude that $g_N$ satisfies the conditions
\eqref{eqGC1}.

To complete the proof
we will show that if $N$
is sufficiently large, then $g_N$ does not
satisfy \eqref{eqGC2}. Indeed, we may assume
that $\eps < \frac1{2} |v|$ which ensures
that the terms in the sum
\eqref{eqGN1} have pairwise disjoint supports.
This implies that
\begin{equation}
\label{eqGN4}
g_N(0) = \varphi(0) \psi(0),
\end{equation}
so that the value $g_N(0)$ does not depend on $N$. 
On the other hand, using 
\eqref{eqGN3}
we have
\begin{equation}
\label{eqGC5}
\int_{\R^d} g_N(x) dx = \ft{g}_N(0) 
\geq \ft{\varphi}(0) \ft{\psi}(0) p_N(0) = N,
\end{equation}
which can be arbitrarily large,
contradicting \eqref{eqGC2}. We have thus
arrived at the desired
  contradiction and so \thmref{thmNCWT}
  is proved.
\end{proof}


\section{Weak tiling and product domains}

In this section we  prove \thmref{thmPRDWKT}.
That is, we  show  that if $A \sbt \R^n$ 
 and $B \sbt \R^m$ are two bounded,
 measurable sets, then 
 	the  product set $\Omega=A\times B$   weakly tiles
 its complement in $\R^n \times \R^m$ by translations
 if and only if  both $A$ and $B$ can weakly tile
 their complements in $\R^n$ and $\R^m$
 respectively.

\begin{proof}[Proof of \thmref{thmPRDWKT}]
Suppose first that  both $A$ and $B$  weakly tile
 their complements. Then there exist two
   positive, locally finite measures $\nu_1$
   on $\R^n$ and $\nu_2$ on $\R^m$, such that
   \begin{equation}
   \text{$\1_{A} \ast (\delta_0 + \nu_1) = 1$ a.e.\  in $\R^n$},
   \quad
   \text{$\1_{B} \ast (\delta_0 + \nu_2) = 1$ a.e.\  in $\R^m$}.
   \end{equation}
It follows that the (also locally finite) product measure 
\begin{equation}
\mu := (\delta_0 + \nu_1) \times(\delta_0 + \nu_2)
\end{equation}
satisfies $\1_{\Om} \ast \mu = 1$ a.e.\ in $\R^n \times \R^m$.
Let $\nu$
be the positive measure
 \begin{equation}
\nu :=   \nu_1  \times \delta_0 + \delta_0 \times \nu_2
+ \nu_1 \times \nu_2,
\end{equation}
and observe that we have $\mu = \delta_0  + \nu$.
Hence 
$\1_{\Om} \ast \nu = \1_{\Om^\complement}$ a.e.\ and
$\Om$ weakly tiles its complement 
in $\R^n \times \R^m$ by translations.

Conversely, suppose that $\nu$
is a positive, locally finite measure in $\R^n \times \R^m$
satisfying
\begin{equation}
\label{eqWTCPRD}
 \1_{\Om} \ast \nu = \1_{\Om^\complement} \quad \text{a.e.}
\end{equation}
For each $y \in \R^m$ we define a measure
$\nu_y$ on $\R^n$ by the condition
\begin{equation}
    \int_{\R^n} \varphi(u) \, d\nu_y(u)
    = \iint_{\R^n \times \R^m} \varphi(u)
    \1_{B}(y-v) \, d \nu(u,v)
\end{equation}
for any bounded, compactly supported Borel
function $\varphi$ on $\R^n$.
It follows that
\begin{align}
& (\1_A \ast \nu_y)(x) = 
    \int_{\R^n} \1_{A}(x-u)  \, d\nu_y(u) \\[14pt]
    & \qquad  = \iint_{\R^n \times \R^m} \1_{A}(x-u) 
    \1_{B}(y-v) \, d \nu(u,v) = (\1_{\Om} \ast \nu)(x,y).
\end{align}
We use this together with \eqref{eqWTCPRD} to conclude
by an application of Fubini's theorem 
that there is a set $Y \sbt \R^m$
of full measure, such that for any $y \in Y$
 we have 
\begin{equation}
(\1_A \ast \nu_y)(x) = \1_{\Om^\complement}(x,y),
\quad x \in X(y),
\end{equation}
where $X(y)$ is a set of full measure in $\R^n$.
In particular, for  $y \in Y \cap B$ this yields
$\1_A \ast \nu_y = \1_{A^\complement}$ a.e.\ in $\R^n$,
which shows that $A$   weakly tiles its
complement  in $\R^n$. In the same way one can show that
also $B$  weakly tiles its   complement
in $\R^m$.
\end{proof}


\section{Packing regions}
\label{secPACKREG}

In this section we introduce the notion of
 a \emph{packing region}, and use it
to establish a geometric condition (different
from the weak tiling condition) that is  necessary  
for the spectrality of a set $A \subset \R^d$.
To prove our result (\thmref{thmA6})
 it will not suffice though to invoke \thmref{thmLMSWT},  since the 
 proof requires additional properties of the weak tiling 
 measure that are established in \cite{LM22}
  but not stated in 
 \thmref{thmLMSWT}.

\subsection{}
In order to define the notion of a packing region, we first
need to introduce the following definition:

\begin{definition}
\label{defESSDIF}
If $A \subset \R^d$ is a bounded, measurable set, then the set
\begin{equation}
	\label{eqC2.1}
	\Delta(A) := \{t \in \R^d: \mes(A \cap (A + t)) > 0\}
\end{equation}
will be called 
the \emph{essential difference set} of $A$.
\end{definition}

The set $\Delta(A)$ is a bounded open set, symmetric with respect to the origin.

The essential difference set
$\Delta(A)$ can be viewed as  the measure-theoretic analog of the algebraic
difference set $A-A$.  In particular, if $A$ is an 
open set then  $\Delta(A) = A-A$.
In general we  have $\Delta(A) \subset A-A$, but this inclusion can be strict.

\begin{definition}
\label{defPACKR}
We say that a bounded, measurable set $D \subset \R^d$
is a \emph{packing region for $A$} if we have
$\Delta(D) \subset \Delta(A)$. 
\end{definition}

Let us explain the  reason for the name ``packing region''.
If $\Lam$ is a finite or countable set in $\R^d$,
then we say that $A + \Lam$ is a \emph{packing}
if the translated copies $\{A + \lam\}$, $\lam \in \Lam$,
are disjoint up to measure zero. Then 
a bounded, measurable set $D$ 
is a packing region for $A$  if and only if
whenever we have that $A + \Lambda$ is a packing,
then also $D+\Lambda$ is a packing.

\subsection{}
The following theorem is the main
result of this section.

\begin{thm}
\label{thmA6}
Let $A$ be a bounded, measurable set  in $\R^d$. Suppose that
$A$ admits a packing region $D$ 
such that $m(D) > m(A)$. Then $A$ is neither spectral nor 
can it tile the space by translations.
\end{thm}

\begin{proof}
We first show 
that $A$ cannot tile by translations. Indeed, if $A + \Lam$ 
is a tiling then $D + \Lam$ must be a packing, and it follows
e.g.\ from \cite[Lemma 3.2]{GL20} that $m(D) \leq m(A)$.
But we assumed that $m(D) > m(A)$, so we arrive at a contradiction.

Next we show 
that $A$ can neither be spectral. Suppose to the contrary that $A$ is spectral, and
let $\gamma$ be the measure given by \cite[Theorem 3.1]{LM22}.
Define the function $f :=  |\ft{\1}_D|^2$, then the Fourier transform of $f$ is the
function  $\ft{f} =  \1_D \ast \1_{-D}$, that is, we have
$\ft{f}(x) = m(D \cap (D+x))$ for every $x \in \R^d$.
In particular,
$\ft{f}(0) = m(D)$, and $\ft{f}$ vanishes on the set $\Delta(D)^\cm$.
On the other hand, we
 have $\ft{\gamma} = m(A) \, \delta_0$ in the open set $\Delta(A)$,
so due to the assumption that $\Delta(D) \subset \Delta(A)$ 
this implies that
$\ft{\gamma} \cdot \ft{f} = m(A) m(D) \, \delta_0$.
The measure $\ft{\gamma} \cdot \ft{f}$ is the Fourier transform of the function
$\gamma \ast f$ (see \cite[Section 2.3]{KL21}),
 and it follows that $\gamma \ast f = m(A) m(D)$ a.e.

Recall now that $\gamma$ is a positive measure, and that
$\gamma = \delta_0$ in some open ball centered at the origin. Since $f$ is a nonnegative function, this implies
that
\[
f = \delta_0 \ast f  \leq \gamma \ast f = m(A) m(D) \quad \text{a.e.}
\]
But since $f$ is a continuous function, it follows that the
inequality $f(t) \leq m(A) m(D)$ must in fact hold
for every $t \in \R^d$. However
 $f(0) = m(D)^2$, so this is possible only if $m(D) \leq m(A)$.
Since we assumed that $m(D) > m(A)$, we arrive at
a contradiction.
\end{proof}

\begin{remark}
Let us comment that the proofs of Theorems \ref{thmNCWT}
and \ref{thmA6}  share a common idea: in order to arrive at a contradiction
we construct a positive definite function with prescribed support,
whose integral is ``too large'' compared to its value at the origin.
\end{remark}

\begin{remark}
We note that there is a notion of an \emph{orthogonal packing region}, 
that was introduced  in \cite{LRW00} and
 extended to general bounded, measurable sets in \cite[Section 3.6]{GL20}.
 By definition, a bounded measurable set  $D \subset \R^d$ is  an 
 ortho\-gonal packing region for $A$ if we have
 $\Delta(D)  \sbt {\zoft{A}}^\complement$.
 This notion was used  in \cite[Lemma 2.3]{LRW00} and
 \cite[Theorem 7]{Kol00b} to establish 
 a result analogous to  \thmref{thmA6},
 namely, if $A$
admits an orthogonal packing region $D$ 
such that $m(D) > m(A)^{-1}$, then $A$ is neither spectral nor 
can it tile by translations.
Notice however that the
 notion of an orthogonal packing region 
 is not purely geometric, as it involves
the set $\zoft{A}$ of the zeros of 
 the Fourier transform $\ft{\1}_A$.
\end{remark}

\subsection{}
We give two examples demonstrating how to use  \thmref{thmA6}.

\begin{example}
Let $A$ and $D$ be the two planar domains
 illustrated respectively
on the left and right hand sides of Figure \ref{fig:PERP}.
It is straightforward to verify that $\Delta(D)$ is the
open rectangle $(-2, 2) \times (-3,3)$, and that
$\Delta(A)$ contains this rectangle. Hence
$D$ serves as a packing region for $A$. But notice that
$m(A) = 5$, $m(D) = 6$, so that $m(D)>m(A)$.
It thus follows from \thmref{thmA6} that $A$ is not spectral.
\end{example}


\definecolor{myblue}{rgb}{0.2,0.2,0.5}
\definecolor{myred}{rgb}{0.9,0.2,0.2}
\tikzstyle{mystyle}=[line width=0.3mm, myblue]

\begin{figure}[ht]
\centering

\begin{tikzpicture}[scale=0.925, style=mystyle]


\fill[myblue,opacity=0.275] 
	(0,0) -- (3,0) -- (3,1) -- (2,1) -- (2,3) -- 
	(1,3) -- (1,1) -- (0,1) -- (0,0);

\draw[gray,opacity=0.2] (-1,-1) grid (4,4);

\draw[myblue]
	(0,0) -- (3,0) -- (3,1) -- (2,1) -- (2,3) -- 
	(1,3) -- (1,1) -- (0,1) -- (0,0);


\fill[myred,opacity=0.275] 
	(7,0) -- (9,0) -- (9,3) -- (7,3) -- (7,0);

\draw[gray,opacity=0.2] (6,-1) grid (10,4);

\draw[myred]
	(7,0) -- (9,0) -- (9,3) -- (7,3) -- (7,0);

\end{tikzpicture}

\caption{The planar domain $A$ shown on the left is not spectral,
since its has  a packing region $D$  (shown on the right) satisfying
 $m(D)>m(A)$.}
\label{fig:PERP}
 \end{figure}
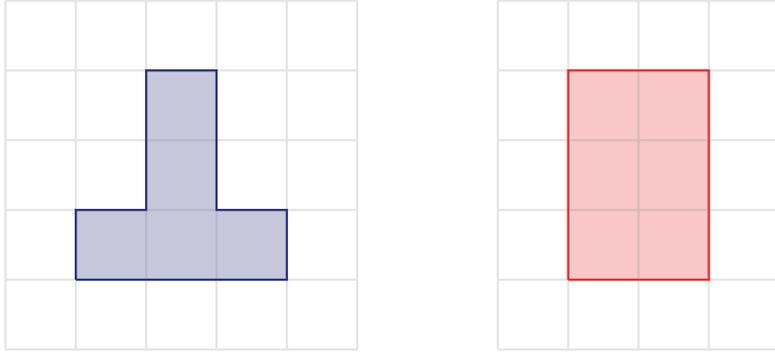


\begin{example}
\label{expNONSYM}
Let $A$ be a convex body  in $\R^d$, and assume
that $A$ is not 
centrally symmetric. Then
the convex body $D = \frac1{2} (A-A)$
is a packing region for $A$, and we have
 $m(D)>m(A)$
 by the Brunn-Minkowski inequality.
Using \thmref{thmA6} this yields
that a spectral 
 convex body must be
centrally symmetric,
a result first proved in \cite{Kol00a} 
based on the same consideration.
\end{example}

\begin{remark}
Due to \cite{LM22}  we know  that a convex body $A \sbt \R^d$
 is  spectral if and only if it satisfies the following four conditions:
  (i)  $A$ is a convex polytope; 
  (ii)  $A$ is centrally symmetric; 
  (iii) all the facets of $A$ are centrally symmetric; and
  (iv) each belt of $A$ has either $4$ or $6$ facets. 
  \exampref{expNONSYM} shows that
  by an application of  \thmref{thmA6} one can prove the
  necessity of condition (ii)   for the spectrality of $A$.
  We observe however that the same is not true for the other three
  conditions (i), (iii) and (iv).
    Indeed, let a convex body $A \sbt \R^d$ be centrally symmetric,
    say, $-A = A$. Then 
$\Delta(A) = 2 \interior(A)$,
where $\interior(A)$ is the set of interior points of $A$. 
If  $D$ is a  packing region for $A$, then $\Delta(D) \subset \Delta(A)$
and hence
\begin{equation}
    \label{eqCVBDPR}
    m(\Delta(D)) \leq m(\Delta(A)) = 2^d m(A).
\end{equation}
On the other hand,
we can invoke a version 
of the   Brunn-Minkowski inequality
for the essential difference set $\Delta(D)$, 
 see the inequality (2.6) in \cite{BL76},
 which implies that
\begin{equation}
    \label{eqSBMI}
    m(\Delta(D)) \geq 2^d m(D).
\end{equation}
It thus follows from \eqref{eqCVBDPR} and \eqref{eqSBMI} 
that   $m(D) \leq m(A)$, 
whenever  $D$ is a  packing region for $A$.  Hence,
 assuming that one of the conditions
 (i), (iii) or (iv) fails to hold does not imply
 the existence of a packing region $D$ 
such that $m(D) > m(A)$, and therefore one cannot
establish the non-spectrality of $A$ by an application of
 \thmref{thmA6}.
\end{remark}

%
%
%
%
%
%

\subsection{}
The next result connects the notions of a packing region and weak tiling.

\begin{thm}
\label{thmA7}
Let $A$ be a bounded, measurable set  in $\R^d$. Suppose that
$A$ admits a packing region $D$ 
such that not only $m(D) > m(A)$, but moreover $D \supset A$.
 Then $A$ cannot weakly tile its complement by translations.
\end{thm}

\begin{proof}
We use the same
 argument as in \cite[Theorem 3.5]{LM22}.
Suppose to the contrary that 
$ \1_{A} \ast \nu = \1_{A^\complement}$ a.e.,
where $\nu$  is  a positive, locally finite measure on $\R^d$.
 Then 
 \begin{equation}
\label{eqMESDPHI}
m(D \cap A^\complement) = \int_{D} \1_{A^\complement} \, dm
= \int_{D} (\1_A \ast \nu) \, dm = \int \varphi \, d\nu,
\end{equation}
where $\varphi(t) := m((A+t) \cap D)$. 
The assumptions that
$D$ is a packing region for $A$ and $D \supset A$
imply that
$\varphi$ must vanish on the set $\Delta(A)^\complement$.
On the other hand, the measure
$\nu$ is supported on the set $\Delta(A)^\complement$
due to \cite[Corollary 2.6]{LM22}. Hence the integral on the right hand side
of \eqref{eqMESDPHI} vanishes. We conclude that
$m(D \cap A^\complement)=0$, or equivalently,
$m(D) = m(D \cap A)$. But this contradicts our assumption
that $m(D) > m(A)$. 
\end{proof}


\section{Cantor sets of positive measure: via the essential difference set}
\label{secCANTOR}

In this section we prove  \thmref{thmCSNWT},
which asserts that  a symmetric Cantor set
of positive measure in $\R$
cannot weakly tile its   complement by translations.
The approach is based on the construction of 
a packing region $D$ for the set $E$ with
$m(D) > m(E)$ and moreover $D \supset E$.
The conclusion then follows from \thmref{thmA7}.

\subsection{}
We start by recalling
the definition of a symmetric Cantor set in $\R$.
Let $\{\xi_k\}$, $k =1,2,\dots$,  be a sequence
of real numbers such that
$0 < \xi_k < \frac{1}{2}$ for all $k$.
The \emph{symmetric Cantor set associated to the
ratio sequence $\{\xi_k\}$} is the closed set
$E(\xi_1, \xi_2, \dots) \subset \R$
obtained from the interval $[0,1]$
by a Cantor type construction, where at the 
$k$'th step we remove from each one of the 
intervals of the previous step a central 
open interval of relative length $1 - 2 \xi_k$,
so that at the $k$'th step we obtain 
$2^k$ closed intervals of common length
$\xi_1 \cdots \xi_k$.

The Lebesgue measure of the set
$E(\xi_1, \xi_2, \dots)$ is
\begin{equation}
    \label{eqMESCS}
    m(E(\xi_1, \xi_2, \dots)) = 
    \lim_{k \to\infty} 2^{k} \xi_1 \cdots \xi_k.
\end{equation}
We are interested in the case where 
$E(\xi_1, \xi_2, \dots)$ 
is a set of positive measure.

\subsection{}
Our proof of  \thmref{thmCSNWT}
 is based on the following lemma:

\begin{lem}
\label{lemFCSWTN}
Let $E$ be a symmetric Cantor set
  of positive measure in $\R$. 
Then $E$ admits a packing region $D$ 
satisfying the two conditions 
$D \supset E$ and $m(D) > m(E)$.
 \end{lem}

  \thmref{thmCSNWT}  follows as an immediate consequence of
\thmref{thmA7} and     \lemref{lemFCSWTN}. So  it remains to prove the lemma.

We note that there exist results in the literature 
which concern the set of differences $E - E$ of a
 symmetric Cantor 
set $E$, see e.g.\ \cite{FN23} and the references
therein. For example, it is well known that if
 $E$ is the classical ternary Cantor set 
(associated to the constant ratio sequence 
$\xi_k = \frac{1}{3}$) then $E-E = [-1,1]$.
However, such results do not suffice for our
present purpose,
as in order to prove \lemref{lemFCSWTN}
we must establish that the \emph{essential difference set}
$\Delta(E)$ (and not only $E-E$)
is quite large.

We now turn to the  proof of \lemref{lemFCSWTN}. This will be done in several steps.

\subsection{}
We will need the following proposition.

\begin{prop}
    \label{propDFUTF}
    Let  $A$ be a bounded, measurable set in $\R^d$,
    and let $B = \bigcup_{j=1}^{n} (A+t_j)$
    where $t_1, \dots, t_n \in \R^d$. Then
    \begin{equation}
        \Delta(B) = \bigcup_{i,j} (\Delta(A) + t_i - t_j).
    \end{equation}
\end{prop}

The proof is straightforward and we omit the details.

\subsection{}
Let
$E(\xi_1, \xi_2, \dots)$ be the
symmetric Cantor set associated to a given
ratio sequence $\{\xi_k\}$
(that is, a sequence
of real numbers such that
$0 < \xi_k < \frac{1}{2}$ for all $k$).
Then the set $E(\xi_1, \xi_2, \dots)$  
consists of all points $x$ of the form
\begin{equation}
    \label{eqEPT}
    x = \sum_{k=1}^{\infty} r_k \eps_k,
\end{equation}
where
\begin{equation}
    \label{eqdefRK}
    r_k := \xi_1 \cdots \xi_{k-1} (1-\xi_k)
\end{equation}
and each $\eps_k$ is either $0$ or $1$.
If we let $T_k$ denote
the finite set of points $x$ of the form
\begin{equation}
    \label{eqEKDF}
    x = \sum_{j=1}^{k} r_j \eps_j 
\end{equation}
where
$\eps_j = 0$ or $1$, then we have
\begin{equation}
    \label{eqESSSTR}
    E(\xi_1,\xi_2,\dots) = T_k + 
    \xi_1 \cdots \xi_k E(\xi_{k+1}, \xi_{k+2}, \dots),
\end{equation}
so that $E(\xi_1,\xi_2,\dots)$ is the union of
$2^k$ disjoint homothetic copies of 
$E(\xi_{k+1}, \xi_{k+2}, \dots)$.

It follows from
\propref{propDFUTF}
that 
\begin{equation}
    \label{eqDESSSTR}
    \Delta(E(\xi_1,\xi_2,\dots)) = T_k - T_k + 
    \xi_1 \cdots \xi_k \Delta(E(\xi_{k+1}, \xi_{k+2}, \dots)).
\end{equation}

\subsection{}
The following result establishes that
\lemref{lemFCSWTN} holds in the special case
where the set $E(\xi_1, \xi_2, \dots)$ 
has measure sufficiently close to $1$.

\begin{lem}
    \label{lemEDFEQQ}
    Suppose that  
    $E(\xi_{1}, \xi_{2}, \dots)$
    has measure at least
    $\frac{4}{5}$. Then
\begin{equation}
    \label{eqEDFEQQ}
    \Delta(E(\xi_{1}, \xi_{2}, \dots)) = \Delta(I),
\end{equation}
where $I := [0, 1]$.
\end{lem}

It follows from \eqref{eqEDFEQQ} that the set
$D := I$ serves as a packing region for
    $E := E(\xi_{1}, \xi_{2}, \dots)$,
and we have $D \supset E$ and 
$m(D) > m(E)$.
We note that  $\Delta(I)$ is simply the interval $(-1,1)$.

\begin{proof}[Proof of \lemref{lemEDFEQQ}]
Write $m(E(\xi_{1}, \xi_{2}, \dots)) = 1 - \eps$.
From  \eqref{eqMESCS} it follows that 
\begin{equation}
    \label{eqTWOXIEST}
    2\xi_1 \geq 1-\eps,
\end{equation}
and
\begin{equation}
    \label{eqMESTAILEST}
    \lim_{n \to\infty} 2^{n} \xi_{k+1} \cdots \xi_{k+n} \geq 1-\eps
\end{equation}
for every $k$.
Note that 
$0 < \eps \leq \tfrac{1}{5}$
by assumption.

The set
$\Delta(E(\xi_{1}, \xi_{2}, \dots))$
is symmetric with  respect to  the  origin 
and it is a subset of $(-1,1)$,
so to establish \eqref{eqEDFEQQ}
it would  be enough to 
 show that  
 $\Delta(E(\xi_{1}, \xi_{2}, \dots))$
contains the interval $[0,1)$.
This will be done in two steps.

\emph{Step 1}:
We first show that 
$\Delta(E(\xi_{1}, \xi_{2}, \dots))$
contains the interval $[0,1-\xi_1)$.

Indeed, let   $t \in [0,1-\xi_1)$.
For any two bounded, measurable sets
$E$ and $F$ we have
    \begin{equation}
        \label{eqMEETFFT}
        m(E \cap (E+t)) \geq   m(F \cap (F+t))- 2 m(E \triangle F),
    \end{equation}
where 
$E \triangle F$ is the symmetric difference of $E$ and $F$. 
We apply this estimate 
to the two sets
$E = E(\xi_{1}, \xi_{2}, \dots)$  and
$F = [0, 1]$.
On one hand, we have 
$F \cap (F + t) =  [t, 1]$
and thus
\begin{equation}
\label{eqMFFT}
 m(F \cap (F + t)) = 1 - t > \xi_1 
 \geq \tfrac{1}{2} -   \tfrac{1}{2} \eps,
\end{equation}
where the last inequality follows from \eqref{eqTWOXIEST}.
On the other hand, notice that
\begin{equation}
    \label{eqSYMDIFFEF}
    m(F \triangle E)
    =   1 - m(E) = \eps.
\end{equation}
So  \eqref{eqMEETFFT},  \eqref{eqMFFT}
and \eqref{eqSYMDIFFEF} imply that
\begin{equation}
m(E \cap (E + t)) > \tfrac{1}{2} 
- \tfrac{5}{2}  \eps \geq 0,
\end{equation}
hence $E \cap (E  + t)$
  is a set of positive measure
  and  $t \in \Delta(E)$.

\emph{Step 2}:
Next we   show that 
$\Delta(E(\xi_{1}, \xi_{2}, \dots))$
contains also the interval $[1-\xi_1, 1)$.

Let therefore   $t \in [1-\xi_1, 1)$, then
we have $r_1 \leq t < 1$. 
Since 
$\sum_{j=1}^{\infty} r_j  = 1$,
 there is $k$ such that
\begin{equation}
    r_1 + \dots + r_k \leq t < 
    r_1 + \dots + r_k + r_{k+1}.
\end{equation}
Then we can write
\begin{equation}
    t = r_1 + \dots + r_k  + \xi_1 \cdots \xi_k s,
    \quad
    0 \leq s < 1-\xi_{k+1}.
\end{equation}
The set 
$E(\xi_{k+1}, 
\xi_{k+2}, \dots)$ also has measure
at least $\frac{4}{5}$, 
due to \eqref{eqMESTAILEST}, so
by what we have already shown
in Step 1
we have $s \in \Delta(E(\xi_{k+1}, 
\xi_{k+2}, \dots))$. Moreover, 
the number $r_1 + \dots + r_k $ belongs to the set
$T_k - T_k$. Hence using  
\eqref{eqDESSSTR}
we conclude that
\begin{equation}
    t \in 
     T_k - T_k + 
    \xi_1 \cdots \xi_k \Delta(E(\xi_{k+1}, \xi_{k+2}, \dots))
    = \Delta(E(\xi_1,\xi_2,\dots)),
\end{equation}
which completes the proof.
\end{proof}

\subsection{}
Let $E_k(\xi_1, \dots, \xi_k)$ be the set 
consisting of the $2^{k}$ disjoint closed intervals
of common  length
$\xi_1 \dots \xi_k$ obtained after the $k$'th
step of the Cantor type construction.
Then 
\begin{equation}
    \label{eqSETEKDEF}
    E_k(\xi_1, \dots, \xi_k) = T_k + 
    \xi_1 \cdots \xi_k I
\end{equation}
where (as before) $I = [0,1]$,
and we have
\begin{equation}
    \label{eqEKINTERS}
    E(\xi_1,\xi_2,\dots) =
    \bigcap_{k=1}^{\infty}
    E_k(\xi_1, \dots, \xi_k).
\end{equation}

It follows from \propref{propDFUTF}
that 
\begin{equation}
    \label{eqDEK}
    \Delta(E_k(\xi_1,\dots,\xi_k)) = T_k - T_k + 
    \xi_1 \cdots \xi_k \Delta(I).
\end{equation}

The next result  concludes the proof of  \thmref{thmCSNWT}
by establishing  that \lemref{lemFCSWTN} holds in the general case.

\begin{lem}
    \label{lemEKDFTEQ}
    Let
$E(\xi_1, \xi_2, \dots)$ have positive measure. Then
    there is $k$ such that 
    \begin{equation}
        \label{eqEKDFTEQ}
        \Delta(E(\xi_1, \xi_2, \dots)) = \Delta(E_k(\xi_1, \dots, \xi_k)).
    \end{equation}
\end{lem}

It follows from \eqref{eqEKDFTEQ} that the set
$D := E_k(\xi_1, \dots, \xi_k)$ serves as a packing region for 
   $E := E(\xi_{1}, \xi_{2}, \dots)$. Moreover, the two
conditions  $D \supset E$ and 
$m(D) > m(E)$ are   satisfied,
so  \lemref{lemFCSWTN} is indeed established.

\begin{proof}[Proof of \lemref{lemEKDFTEQ}]
We can choose $k$ sufficiently large such that
    \begin{equation}
    \label{eqSMTL}
    \lim_{n \to\infty}
    2^{n} \xi_{k+1} \xi_{k+2} \cdots \xi_{k+n} \geq \tfrac{4}{5}
\end{equation}
(for otherwise, there would exist a sequence
$k_1 < k_2 < k_3 < \cdots$
such that
\begin{equation}
2^{k_{j+1} - k_j} \xi_{k_j+1} \xi_{k_j+2} \cdots \xi_{k_{j+1}} < \tfrac{4}{5}
\end{equation}
for every $j$, which implies that the limit in
    \eqref{eqMESCS} is zero contrary to our assumption).

Then by \lemref{lemEDFEQQ} we have
\begin{equation}
    \label{eqETLDF}
    \Delta(E(\xi_{k+1}, \xi_{k+2}, \dots)) = \Delta(I).
\end{equation}
Together with \eqref{eqDESSSTR}
and \eqref{eqDEK} this implies that
\begin{equation}
    \label{eqFUOI}
    \Delta(E(\xi_1,\xi_2,\dots)) = T_k - T_k + 
    \xi_1 \cdots \xi_k \Delta(I) =
    \Delta(E_k(\xi_1,\dots,\xi_k)),
\end{equation}
and  thus \eqref{eqEKDFTEQ} is proved.
\end{proof}

\begin{remark}
It follows from \eqref{eqFUOI} that if
$E(\xi_1, \xi_2, \dots)$ has positive measure,
then the essential difference
set $\Delta(E(\xi_1,\xi_2,\dots))$ is the union
of finitely many open intervals.
\end{remark}


\section{Cantor sets of positive measure: via weak tiling}
\label{secOVERLAPS}

In this section we give another proof of Theorem \ref{thmCSNWT}, which states 
that a symmetric Cantor set of positive measure cannot weakly
tile its complement by translations. In this proof the approach is direct
and does not go through the essential difference set.

\subsection{}
Suppose that we have a Cantor set $E \subset \RR$ which is the intersection of the compact sets $E_n$, each of which is a finite collection of closed intervals of length $\ell_n$. The set $E_{n}$ is constructed from $E_{n-1}$ by throwing away an open middle interval of length $d_n$ from each of the intervals of $E_{n-1}$. For $E$ to have positive measure it is \textit{necessary} that $d_n/\ell_n \to 0$.

The idea of this proof is to use the fact that the sets $E_n$ obviously cannot weakly tile their complement (the holes are too small compared to the intervals making up $E_n$). Making this observation quantitative is the key.

Suppose $\mu = \delta_0+\nu$, with $\nu$ being a nonnegative Borel measure and
\beql{wtiling}
\one_E \ast \mu = 1 \text{ on $\RR$ so also } \one_E \ast \nu = 0 \text{ on $E$}.
\eeq
Write $A_n$ for the union of intervals (each of length $d_n$) that we threw away from the intervals of $E_{n-1}$ in order to obtain the intervals of $E_n$ (each of length $\ell_n$).

Observe that
\beql{ratio}
m(A_n) = \frac{d_n}{2\ell_n} m(E_n).
\eeq

Since $E_{n+1} \subset E_n$ we have by the monotone convergence theorem that
\beql{lim}
\int_{E_n} \one_{E_n} \ast \nu \to \int_E \one_E \ast \nu = 0, \text{ as $n \to \infty$.}
\eeq
The last equality is due to \eqref{wtiling}.

We also have the crucial inequality (for all $n$ for which $d_n/\ell_n < 1$)
\beql{ineq}
\frac{\ell_n-d_n}{d_n} \int_{A_n} \one_{E_n} \ast \nu \le \int_{E_n} \one_{E_n} \ast \nu.
\eeq
To see this note that every time we ``load'' an interval of $A_n$ with a fractional copy (via the measure $\nu$) of some interval of $E_n$ we also ``load'' one or both the intervals of $E_n$ on each side of $A_n$ by at least a multiple $(\ell_n-d_n)/d_n$ of the load that goes onto the $A_n$-interval.

More formally, we first observe that for each interval $J$ of $E_{n-1}$ and for all $t \in \RR$ we have
\beql{J}
\frac{\ell_n-d_n}{d_n} m((E_n+t) \cap A_n \cap J) \le m((E_n+t) \cap E_n \cap J).
\eeq
Summing over all intervals $J$ comprising $E_{n-1}$ we obtain
\beql{noJ}
\frac{\ell_n-d_n}{d_n} m((E_n+t) \cap A_n) \le m((E_n+t) \cap E_n)
\eeq
or, equivalently,
\beql{noJ2}
\frac{\ell_n-d_n}{d_n} \int_{A_n} \one_{E_n}(x-t)\,dx \le \int_{E_n} \one_{E_n}(x-t)\,dx,
\eeq
and integration $d\nu(t)$ followed by Fubini's theorem gives \eqref{ineq}.

We continue from \eqref{ineq} using \eqref{ratio}:
\begin{align*}
\int_{E_n} \one_{E_n} \ast \nu &\ge \frac{\ell_n-d_n}{d_n} \int_{A_n} \one_{E_n} \ast \nu\\
 &\ge \frac{\ell_n-d_n}{d_n} \int_{A_n} \one_{E} \ast \nu \text{\ \  (since $E \subset E_n$)} \\
 &= \frac{\ell_n-d_n}{d_n} m(A_n) \text{\ \  (due to \eqref{wtiling} since $A_n \subset E^{\complement}$)}\\
 &= \frac{\ell_n-d_n}{d_n} \frac{d_n}{2\ell_n} m(E_n) \text{\ \  (from \eqref{ratio})}\\
 &= \left(\frac12 - \frac{d_n}{2\ell_n}\right) m(E_n)\\
 &\to \frac12 m(E) \text{ as $n\to\infty$.}
\end{align*}
This positive lower bound contradicts the limit \eqref{lim} and finishes the proof.

\subsection{}
It should be apparent that the proof above is quite flexible and does not impose much rigidity on the Cantor sets of positive measure to which it applies. Instead of trying to state the most general result possible let us indicate this flexibility by giving an example in two dimensions, to which the method applies. This Cantor set is not a cartesian product.

Define a Cantor set $E \subset [0, 1]^2$ of positive measure as follows (refer to Figure \ref{fig:cubes}).

The set $E$ will be the intersection of the decreasing sequence of compact sets $E_n$, with $E_0 = [0, 1]^2$. The $n$-th stage set $E_n$ will be a union of non-overlapping cubes of the same side-length $s_n$, all of them aligned at multiples of $s_n$. To obtain the set $E_n$ from $E_{n-1}$ we visit each of the cubes of side-length $s_{n-1}$ comprising $E_{n-1}$, we subdivide such a cube $Q$ into cubes of side-length $s_n = s_{n-1}/M_n$ (here $M_n>0$ is a fast increasing integer sequence) and we throw away any one of these cubes situated in the middle third of $Q$.

If the integer sequence $M_n$ grows sufficiently fast then the resulting Cantor set has positive measure. The proof of the previous section then applies with no essential changes.

\begin{figure}[ht]
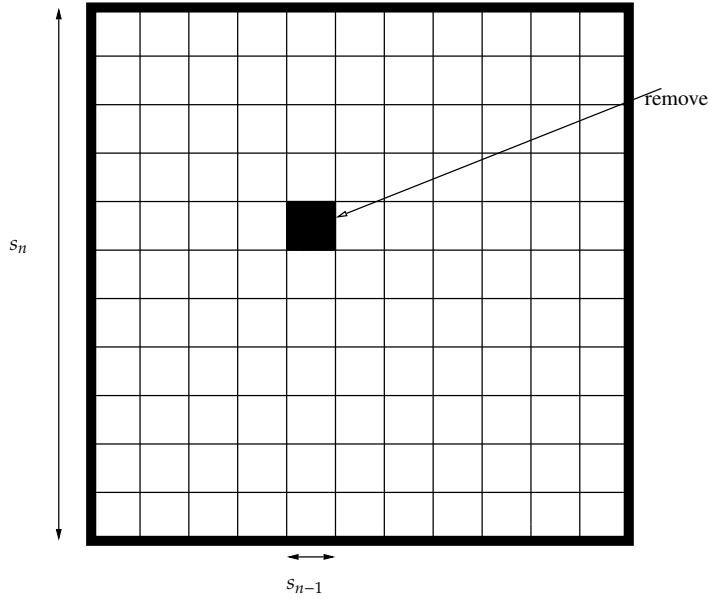


{\pgfkeys{/pgf/fpu/.try=false}%
	\ifx\XFigwidth\undefined\dimen1=0pt\else\dimen1\XFigwidth\fi
	\divide\dimen1 by 4343
	\ifx\XFigheight\undefined\dimen3=0pt\else\dimen3\XFigheight\fi
	\divide\dimen3 by 3591
	\ifdim\dimen1=0pt\ifdim\dimen3=0pt\dimen1=4143sp\dimen3\dimen1
	\else\dimen1\dimen3\fi\else\ifdim\dimen3=0pt\dimen3\dimen1\fi\fi
	\tikzpicture[x=+\dimen1, y=+\dimen3]
	{\ifx\XFigu\undefined\catcode`\@11
		\def\temp{\alloc@1\dimen\dimendef\insc@unt}\temp\XFigu\catcode`\@12\fi}
	\XFigu4143sp
	\ifdim\XFigu<0pt\XFigu-\XFigu\fi
	\catcode`\@11
	\pgfutil@ifundefined{pgf@pattern@name@xfigp3}{
		\pgfdeclarepatternformonly{xfigp3}
		{\pgfqpoint{-1bp}{-1bp}}{\pgfqpoint{9bp}{9bp}}{\pgfqpoint{8bp}{8bp}}
		{	\pgfsetdash{}{0pt}\pgfsetlinewidth{0.45bp}
			\pgfpathqmoveto{-1bp}{9bp}\pgfpathqlineto{9bp}{-1bp}
			\pgfusepathqstroke
		}
	}{}
	\catcode`\@12
	\pgfdeclarearrow{
		name = xfiga1,
		parameters = {
			\the\pgfarrowlinewidth \the\pgfarrowlength \the\pgfarrowwidth\ifpgfarrowopen o\fi},
		defaults = {
			line width=+7.5\XFigu, length=+120\XFigu, width=+60\XFigu},
		setup code = {
			\dimen7 2.1\pgfarrowlength\pgfmathveclen{\the\dimen7}{\the\pgfarrowwidth}
			\dimen7 2\pgfarrowwidth\pgfmathdivide{\pgfmathresult}{\the\dimen7}
			\dimen7 \pgfmathresult\pgfarrowlinewidth
			\pgfarrowssettipend{+\dimen7}
			\pgfarrowssetbackend{+-\pgfarrowlength}
			\dimen9 -\pgfarrowlength\advance\dimen9 by-0.45\pgfarrowlinewidth
			\pgfarrowssetlineend{+\dimen9}
			\dimen9 -\pgfarrowlength\advance\dimen9 by-0.5\pgfarrowlinewidth
			\pgfarrowssetvisualbackend{+\dimen9}
			\pgfarrowshullpoint{+\dimen7}{+0pt}
			\pgfarrowsupperhullpoint{+-\pgfarrowlength}{+0.5\pgfarrowwidth}
			\pgfarrowssavethe\pgfarrowlinewidth
			\pgfarrowssavethe\pgfarrowlength
			\pgfarrowssavethe\pgfarrowwidth
		},
		drawing code = {\pgfsetdash{}{+0pt}
			\ifdim\pgfarrowlinewidth=\pgflinewidth\else\pgfsetlinewidth{+\pgfarrowlinewidth}\fi
			\pgfpathmoveto{\pgfqpoint{-\pgfarrowlength}{-0.5\pgfarrowwidth}}
			\pgfpathlineto{\pgfqpoint{0pt}{0pt}}
			\pgfpathlineto{\pgfqpoint{-\pgfarrowlength}{0.5\pgfarrowwidth}}
			\pgfpathclose
			\ifpgfarrowopen\pgfusepathqstroke\else\pgfsetfillcolor{.}
			\ifdim\pgfarrowlinewidth>0pt\pgfusepathqfillstroke\else\pgfusepathqfill\fi\fi
		}
	}
	\clip(1560,-8726) rectangle (5903,-5135);
	\tikzset{inner sep=+0pt, outer sep=+0pt}
	\pgfsetlinewidth{+7.5\XFigu}
	\pgfsetstrokecolor{black}
	\draw (2346,-5189)--(2346,-8372);
	\draw (2635,-5189)--(2635,-8372);
	\draw (2924,-5189)--(2924,-8372);
	\draw (3214,-5189)--(3214,-8372);
	\draw (3503,-5189)--(3503,-8372);
	\draw (3792,-5189)--(3792,-8372);
	\draw (4081,-5189)--(4081,-8372);
	\draw (4371,-5189)--(4371,-8372);
	\draw (4660,-5189)--(4660,-8276);
	\draw (4949,-5189)--(4949,-8372);
	\draw (4660,-8372)--(4660,-8276);
	\draw (2057,-8083)--(5239,-8083);
	\draw (2057,-7794)--(5239,-7794);
	\draw (2057,-7505)--(5239,-7505);
	\draw (2057,-7215)--(5239,-7215);
	\draw (2057,-6926)--(5239,-6926);
	\draw (2057,-6637)--(5239,-6637);
	\draw (2057,-6347)--(5239,-6347);
	\draw (2057,-6058)--(5239,-6058);
	\draw (2057,-5479)--(5239,-5479);
	\pgfsetlinewidth{+60\XFigu}
	\draw (2057,-5189) rectangle (5239,-8372);
	\pgfsetlinewidth{+7.5\XFigu}
	\pgfsetfillpattern{xfigp3}{black}
	\draw[pattern,preaction={fill=black}] (3214,-6347) rectangle (3503,-6637);
	\pgfsetarrows{[line width=7.5\XFigu, width=26\XFigu, length=51\XFigu]}
	\pgfsetarrows{xfiga1-xfiga1}
	\draw (1865,-5189)--(1865,-8372);
	\draw (3214,-8468)--(3503,-8468);
	\pgfsetarrows{-}
	\draw (5239,-5286)--(5239,-8372);
	\draw (2057,-5769)--(5239,-5769);
	\pgfsetarrows{[open]}
	\pgfsetarrowsend{xfiga1}
	\draw (5432,-5672)--(3503,-6444);
	\pgfsetfillcolor{black}
	\pgftext[base,left,at=\pgfqpointxy{3214}{-8662}] {\fontsize{8}{9.6}\usefont{T1}{ptm}{m}{n}$s_{n-1}$}
	\pgftext[base,left,at=\pgfqpointxy{5335}{-5769}] {\fontsize{8}{9.6}\usefont{T1}{ptm}{m}{n}remove}
	\pgftext[base,left,at=\pgfqpointxy{1575}{-6637}] {\fontsize{8}{9.6}\usefont{T1}{ptm}{m}{n}$s_n$}
	\endtikzpicture}%

\caption{To go from $E_{n-1}$ to $E_n$ we remove one of the small cubes near the center of the big cube.}
\label{fig:cubes}
\end{figure}

\subsection{}
Finally let us mention that there do exist spectral
\emph{unbounded} nowhere dense sets  of positive and finite
 measure. One way to obtain such a set is to construct a set
that tiles by $\Z^d$ translates, and which therefore admits $\Z^d$
also as a spectrum. We describe the construction in dimension one,
but a similar idea works also in several dimensions.

Assume first that we have
\[
\one_{[0, 1]}(x) = \sum_{n \in \ZZ} \one_{E_n}(x) \quad \text{a.e.}
\]
where each $E_n$ is a  nowhere dense subset of $[0,1]$.
It follows then that the union
\[
E = \bigcup_{n\in\ZZ} (E_n + n)
\]
is a nowhere dense (unbounded) set in $\R$ which tiles by $\ZZ$ translates, 
and is therefore spectral.

To construct this partition $E_n$ we first pick a fat Cantor set in $[0, 1]$ (a Cantor set of positive measure -- we can construct such sets of arbitrarily large measure in any given interval) as our first set. At any stage in the construction the complement of the (finitely many) closed sets we have selected so far is an open subset of $(0, 1)$, which is therefore a disjoint countable union of intervals. In each of these intervals we select another fat Cantor set, taking care for the total measure of the complement to go to zero. This process exhausts the measure and it is clear that the resulting set coincides a.e.\ with a fundamental domain of the lattice $\ZZ$ in $\RR$. In other words $E+\ZZ$ is a tiling.


\section{Open problems}

We conclude the paper by posing some open problems.

\subsection{}
Let $\Om \sbt \R^d$ be a bounded, nowhere dense set 
of positive measure. Can $\Om$  be a spectral set?
As we have mentioned in \secref{subsecNW}, the answer is 
known to be negative in dimension one,
but in dimensions two and higher the problem is open.

\subsection{}
Let $\Omega=A\times B$ where $A$ is a convex body in $\R^n$, 
and $B$ is a bounded, measurable set in $\R^m$.
If $\Omega$ is a spectral set, then  $A$ must be spectral
according to \thmref{thmPRDCVX}. Is it true that
also $B$ must be spectral?

At present this is proved only for dimensions
$n=1$ and $2$ \cite{GL16}, \cite{GL20}.

\subsection{}
Let $K$ be a convex body in $\R^d$, and assume that $K$ can
weakly tile its  complement  by translations. Let $W(K)$ be the (nonempty,
convex) set of all positive, locally finite measures $\nu$ such that
$\1_{K} \ast \nu = \1_{K^\cm}$ a.e. If we endow $W(K)$ with the topology of
vague convergence, then $W(K)$ is also compact, so
by the Krein-Milman theorem $W(K)$ is the closed convex
hull of its extremal points. In particular, $W(K)$ has extremal points.

It is not difficult to verify  that any proper tiling 
(that is, any measure $\nu \in W(K)$ which is the sum of unit masses)
is an extremal point of $W(K)$. Is the converse true?

\subsection{}
Let $\Om$ be a bounded, measurable set in $\R^d$.
Consider the following   properties:
\begin{enumerate-roman}
\item
  \label{lemFN1.1.2}
  $\Om$ can tile the space (properly) by translations;
\item
  \label{lemFN1.1.3}
  $\Om$ is spectral;
\item
  \label{lemFN1.1.1}
  $\Om$ can weakly tile its complement by translations.
\end{enumerate-roman}

It is obvious that \ref{lemFN1.1.2} implies \ref{lemFN1.1.1},
and by \thmref{thmLMSWT} we know that also  
 \ref{lemFN1.1.3} implies \ref{lemFN1.1.1}.
On the other hand, 
\ref{lemFN1.1.1} does not imply 
\ref{lemFN1.1.2} (as an example, take a spectral set that cannot tile),
and also 
\ref{lemFN1.1.1} does not imply 
\ref{lemFN1.1.3} (take a tile that is not spectral).

Does there exist a set $\Om$ satisfying
\ref{lemFN1.1.1}, but such that
both \ref{lemFN1.1.2} and
\ref{lemFN1.1.3} do not hold?


\end{document}